\documentclass{article}
\usepackage{amsthm, amsfonts, amssymb,latexsym, tikz}

\usepackage{comment}
\title{New Expander Bounds from Affine Group Energy}
\newcommand{\F}{\mathbb{F}}
\newcommand{\Z}{\mathbb{Z}}
\newcommand{\R}{\mathbb{R}}

\usepackage{amsmath}
\usepackage[a4paper, total={6in, 8in}]{geometry}
\newcommand{\cL}{\mathcal{L}}

\newcommand{\cP}{\mathcal{P}}

\usepackage{fixme}
\fxsetup{
    status=draft,
    author=,
    layout=margin,
    theme=color
}

\author{Oliver Roche-Newton and Audie Warren}
\newtheorem{Theorem}{Theorem}

\newtheorem{Corollary}[Theorem]{Corollary}
\newtheorem{Lemma}[Theorem]{Lemma}

\newtheorem{Conjecture}[Theorem]{Conjecture}
\newtheorem*{Theorem*}{Theorem \ref{pencilstheorem}}
\newtheorem*{Theorem**}{Theorem \ref{thm:FSMP}}
\begin{document}
 \maketitle
 
 \begin{abstract}
    The purpose of this article is to further explore how the structure of the affine group can be used to deduce new incidence theorems, and to explore sum-product type applications of these incidence bounds, building on the recent work of Rudnev and Shkredov \cite{RuSh}.
    
    We bound the energy of several systems of lines, in some cases obtaining a better energy bound than the corresponding bounds in \cite{RuSh} by exploiting a connection with collinear quadruples.
    
    Our motivation for seeking to generalise and improve the incidence bound from \cite{RuSh} comes from possible applications to sum-product problems. For example, we prove that, for any finite $A \subset \mathbb R$ the following superquadratic bound holds:
    \[
    \left| \left \{ \frac{ab-cd}{a-c} : a,b,c,d \in A \right \} \right| \gg |A|^{2+\frac{1}{14}}.
    \]
    This improves on a bound with exponent $2$ that was given in \cite{MRNS}. We also give a threshold-beating asymmetric sum-product estimate for sets with small sum set by proving that there exists a positive constant $c$ such that for all finite $A,B \subset \mathbb R$,
    \[
    |A+A| \ll K|A| \Rightarrow |AB| \gg_K |A||B|^{1/2+c}.
    \]
    
\end{abstract}

 \section{Introduction}\label{intro}

\setlength{\parindent}{2.5em} \setlength{\parskip}{0.7em}
The central topic of this paper is that of point-line incidence bounds in $\mathbb R^2$. Given a set $\cP \subset \mathbb R^2$ and a set $\cL$ of lines in $\mathbb R^2$, the number of incidences between $\cP$ and $\cL$ is defined as
\[
I(\cP, \cL):= |\{(p,l) \in \cP \times \cL : p \in l \}|.
\]
The much celebrated Szemer\'{e}di-Trotter Theorem gives the following\footnote{Here and throughout the paper, the standard notation
$\ll,\gg$ and respectively $O$ and $\Omega$ is applied to positive quantities in the usual way. That is, $X\gg Y$, $Y \ll X,$ $X=\Omega(Y)$ and $Y=O(X)$ all mean that $X\geq cY$, for some absolute constant $c>0$. If both $X \ll Y$ and $Y \ll X$ hold we write $X \approx Y$, or equivalently $X= \Theta(Y)$.}
optimal bound for this quantity:
\begin{equation} \label{SzTr}
I(\cP, \cL) \ll |\cP|^{2/3}|\cL|^{2/3}+|\cP|+|\cL|.
\end{equation}
This result has numerous applications, and the applications most relevant to this paper concern the sum-product problem. In such applications, we are often interested in the case when $\cP=A \times B$ is a Cartesian product, or the line set is defined by a Cartesian product. This connection was first discovered in a beautiful paper of Elekes \cite{E}. See also \cite[Chapter 8]{TV} for introductory material on this topic.

In \cite{RuSh}, the authors prove an incidence theorem for Cartesian products of points $A\times B$ and arbitrary finite sets of non vertical lines $\cL$, in terms of the `energy' of $\cL$. The energy of a set of lines originates from considering the lines as element of the affine group $\text{Aff}(\R)$, via
$$y = mx + c  \longleftrightarrow (m,c) \in \text{Aff}(\R).$$
We can therefore talk interchangeably about (non-vertical, non-horizontal) lines $l$ in the plane and their realisation in the affine group. Multiplication in the affine group is analogous to line composition in the plane, via
$$(a,b) \cdot (c,d) = (ac, ad+b).$$
Let $r_{\cL^{-1} \cL}(l)$ denote the number of pairs $(l_1,l_2) \in \cL \times \cL$ such that $l_1^{-1} l_2 =l $. The energy of $\cL$ is then defined as
$$E(\cL) := \sum_{l \in \text{Aff}(\R)}r_{\cL^{-1} \cL} (l)^2. $$
This energy can be expanded to be considered as the number of solutions to
$$l_1^{-1}l_2 = l_3^{-1}l_4$$
such that each $l_i$ is in $\cL$.


In the Euclidean setting, the fundamental new incidence theorem of \cite{RuSh} was the following:

 \begin{Theorem}{\cite[Theorem 8]{RuSh}}\label{mishailya8}
 Let $A, B \subseteq \R$ be finite sets, and let $\cL$ be a finite set of non-vertical and non-horizontal lines. Then we have
 $$I(A \times B, \cL) \ll |B|^{1/2}|A|^{2/3}E(\cL)^{1/6}|\cL|^{1/3} + |B|^{1/2}|\cL|.$$
  \end{Theorem}
  
  If one can obtain strong bounds for $E(\cL)$ for certain families $\cL$ of lines, Theorem \ref{mishailya8} can be used to give interesting and applicable incidence theorems. In \cite{RuSh}, the authors prove a non-trivial upper bound 
  \begin{equation} \label{RSEnergy}
      E(\cL) \ll |C|^{5/2}|D|^3+ |C|^3|D|^2
  \end{equation}
  when $\cL$ is of the form $\{(c,d): c \in C, d \in D\}$ or $\{(c,cd): c \in C, d \in D\}$ for any finite sets $C \subset \R^*$ and $D \subset \R$. This resulted in the bound
  \begin{equation} \label{RSgrids}
  I(A \times B, \cL) \ll |A|^{2/3}|B|^{1/2}|C|^{3/4}|D|^{5/6}+|B|^{1/2}|C||D|.
  \end{equation}
  This result improves on \eqref{SzTr} in the case when the sizes of the sets $A,B,C$ and $D$ are suitably imbalanced, which gives the potential for interesting new applications, the first of which were explored in \cite{RuSh}.
  
  In this paper we prove bounds on the energy of two other families of line sets, and give an application for each.  We also generalise Theorem \ref{mishailya8} to apply to skew grids. Our main energy theorems are as follows.

\begin{Theorem}\label{energythm1} Let $C,D \subset \mathbb R$ be finite sets, and let $\cL \subset \text{Aff}(\R)$ be the lines of the form $\left(\frac{\lambda}{c-d}, \frac{\mu c}{c-d}\right)$ with $c \in C$, $d \in D$, $c \neq d$, and $\lambda, \mu \in \R \setminus \{0\}$. 
Then we have
$$E(\cL) \ll |C|^{5/2}|D|^{5/2} + |C|^4  + |D|^4.$$\end{Theorem}
\begin{Theorem}\label{energythm2}
 Let $C,D \subseteq \R^*$ be finite sets, and let $\cL \subset \text{Aff}(\R)$ be the lines of the form $( d(c-\lambda) - \mu, c)$ with $c \in C$, $d \in D$, $\lambda, \mu \in \R$ such that $d(c-\lambda) - \mu \neq 0$ . Then we have
$$E(\cL) \ll|C|^3 |D|^{5/2} +  |C|^2 |D|^3.$$
\end{Theorem}

Theorem \ref{energythm2} gives a bound very similar to that obtained in \cite{RuSh} for lines of the form $(c,d)$ or $(c,cd)$, and the proof is also somewhat similar. Theorem \ref{energythm1}, on the other hand, gives a quantitative improvement on the corresponding energy bounds in \cite{RuSh} by exploiting a connection with collinear quadruples. Furthermore, the bound given in Theorem \ref{energythm1} is tight, up to logarithmic factors.


  \subsection{Applications to the Sum-Product Phenomenon}

Our main application of these energy bounds concerns the sum-product phenomenon. Let $A \subseteq \R$ be a finite set. The \emph{sum set} $A+A$ and the \emph{product set}  $AA$ are defined as
$$A+A = \{ a + a' : a, a' \in A \}, \quad AA = \{ aa' : a, a' \in A \}.$$
Other combinations of $A$ are defined analogously. The sum-product phenomenon states that we expect one of these sets to be large with respect to $|A|$. The central conjecture in this area is contained in the seminal paper of Erd\H{o}s and Szemer\'{e}di \cite{ErdSze} from 1983.

\begin{Conjecture}[Erd\H{o}s - Szemer\'{e}di]
Let $A \subset \Z$ be a finite set. Then for all $\epsilon > 0$ we have
$$|A+A| + |AA|\gg |A|^{2 - \epsilon}.$$
\end{Conjecture}

This conjecture states that at least one of $|A+A|$ or $|AA|$ must be large, but even the extremal cases when one of the sets is close to linear in size are not completely settled. It was proven by Elekes and Ruzsa \cite{ER} that
\[|A+A|^4|AA| \gg \frac{|A|^6}{\log|A|},
\]
and in particular it follows that 
\begin{equation} \label{ERsym}
|A+A|\leq K|A| \Rightarrow |AA| \gg \frac{|A|^2}{K^4\log|A|}.
\end{equation}
A version of \eqref{ERsym} with a better dependency on $K$ follows from the beautiful work of Solymosi \cite{So}. 

We consider an asymmetric version of this question, where the ultimate goal would be to prove a result of the form
\begin{equation} \label{ERasym}
|A+A|\leq K|A| \Rightarrow |AB| \gg \frac{|A||B|}{K^C\log|A|},
\end{equation}
for any $B \subset \mathbb R$ and some constant $C$. Such a result would have some striking implications, and appears to be out of reach at present. It follows from \cite {E} (and also from \cite{So}), that
\begin{equation} \label{ERbasic}
|A+A|\leq K|A| \Rightarrow |AB| \gg \frac{|A||B|^{1/2}}{K}.
\end{equation}
So \eqref{ERbasic} represents the `threshold bound' for this problem; that is, the bound that can be deduced straightforwardly from the known methods. We present the following small improvement.

\begin{Theorem} \label{thm:FSMP}
For all $\kappa>0$ there exists $k=k(\kappa)>0$ such that for all $A,B \subset \mathbb R$ with $|A|^{\kappa} \leq |B|$ and writing $|A+A|=K|A|$, we have
\[ |AB| \gg \frac{|A||B|^{1/2+k}}{K^{4/3}}.
\]

\end{Theorem}

In the same spirit as the Erd\H{o}s-Szemer\'{e}di conjecture, we may also consider sets that are defined using a combination of multiplication and addition, which we expect to \emph{always} be large. These sets are often called `expanders' due to the growth they exhibit. Several expander results, including the bound 
\begin{equation} \label{A(A+A+A+A)} 
|A(A+A+A+A)| \gg \frac{|A|^2}{\log |A|},
\end{equation}
were given by Murphy, Roche-Newton and Shkredov in \cite{MRNS}. Note that this quadratic lower bound is optimal up to logarithmic factors, as can be seen by taking $A$ to be an arithmetic progression.

The same paper also included the bound
\begin{equation}
\label{oldint}
\left| \left\{ \frac{ad-bc}{a-c}  : a,b,c,d \in A \right\} \right| \gg |A|^{2}.
\end{equation}
The quantity $\frac{ad-bc}{a-c}$ has some geometric meaning: if we draw a straight line through $(a,b)$ and $(c,d)$ with $a \neq c$, this line will intersect the $y$-axis at the point
$\left (0, \frac{ad-bc}{a-c}\right)$. So the set in \eqref{oldint} is the set of all $y$-intercepts of lines determined by $A \times A$.

In this paper, we give the following improvement to \eqref{oldint}.


\begin{Theorem}\label{expander}
 Let $A \subset \R$ be a finite set. Then
 $$\left| \left\{ \frac{ac-db}{c-d}  : a,b,c,d \in A \right\} \right| \gg |A|^{2 + 1/14}.$$
\end{Theorem}
 This is an example of a four-variable super-quadratic expander. Several six-variable super-quadratic expanders were proven to exist by Balog, Roche-Newton and Zhelezov \cite{BRNZ}. Typically things get more difficult with less variables, and examples of super-quadratic four variable are rare in the literature. We are aware of only two such results, due to Rudnev \cite{Ru} and Shkredov \cite{Sh}.

  \subsection{Structure of this paper}
 The rest of the paper will be structured as follows. In section \ref{sec:energy} we will prove our two bounds for the energies of line sets, Theorems \ref{energythm1} and \ref{energythm2}. In section \ref{sec:geometric} we introduce a new geometric conjecture and use the new energy bounds to prove some special cases. One of these special cases gives Theorem \ref{expander}. Section \ref{sec:FSMP} gives the proof of Theorem \ref{thm:FSMP}. Finally, in section \ref{sec:other} we collect some other applications of the techniques in this paper and also \cite{RuSh}, with the focus on proving new expander results with three variables. The main result of this section is an improved explicit bound for the size of the set $AA+A$.
 

 \section{Energy of Lines} \label{sec:energy}

In this section we prove Theorems \ref{energythm1} and \ref{energythm2}.
We will make use of the following simple corollary of the Szemer\'{e}di-Trotter Theorem concerning the number of $k$ rich lines defined by a finite point set.
\begin{Corollary} \label{krich}
 Let $P \subset \R^2$ be a finite set of points. Let $k \geq 2$ and let $\cL_k(P)$ denote the set of all lines containing at least $k$ points from $P$. Then we have
 $$|\cL_k(P)| \ll \frac{|P|^2}{k^3} + \frac{|P|}{k}.$$
\end{Corollary}
\subsection{Proof of Theorem \ref{energythm1}}
 \begin{proof}[Proof of Theorem 2]
Recall that the aim is to bound the energy of the set of lines 
 \[\cL:= \left \{\left(\frac{\lambda}{c-d}, \frac{\mu c}{c-d}\right): c \in C, d \in D, c \neq d \right \}.
 \]  
 Inverses of lines in $\cL$ have the form $\left(\frac{c-d}{\lambda},-\frac{\mu}{\lambda}c \right)$. The quantity $E(\cL)$ is therefore the number of solutions to the equations
 \begin{equation}\label{Energyvert1}
     \frac{c_1 - d_1}{c_2 - d_2} =  \frac{c_3 - d_3}{c_4 - d_4} 
 \end{equation}
 \begin{equation}\label{Energyvert2}
     \frac{c_1d_2 - c_2d_1}{d_2 - c_2} =    \frac{c_3d_4 - c_4d_3}{d_4 - c_4}.
 \end{equation}
Note that equation \eqref{Energyvert1} asserts that the line connecting the points $(c_2,c_1)$ and $(d_2,d_1)$ has the same slope as the line connecting the points $(c_4,c_3)$ and $(d_4,d_3)$, i.e. they are parallel. Equation \eqref{Energyvert2} then asserts that these two lines in fact have the same $y$-axis intercept. Indeed, upon calculating the equations of the lines connecting the pairs of points above, the $y$ intercepts are precisely $\frac{c_1d_2 - c_2d_1}{d_2 - c_2}$ and $ \frac{c_3d_4 - c_4d_3}{d_4 - c_4}$. Therefore the lines are the same, and we have a collinear quadruple  $(c_2,c_1)$, $(d_2,d_1)$, $(c_4,c_3)$, and $(d_4,d_3)$. Hence the energy $E(\cL)$ is no larger than the number of collinear quadruples $(p_1,p_2,p_3,p_4) \in (C \times C)^2 \times (D \times D)^2$. Our goal is now to bound the number of such quadruples.

We begin with the trivial observation that there are at most $|C|^2|D|^2$ such quadruples with $p_1=p_2$ and $p_3=p_4$. Therefore,
 \begin{equation*}
 E(\cL) \leq |C|^2|D|^2+ \sum_{\text{lines }l: \max \{|l \cap(C \times C)|,|l \cap (D \times D)| \} \geq 2} |l \cap (C \times C)|^2|l \cap (D \times D)|^2 .
 \end{equation*}
We firstly separate this sum as
 \begin{equation} \label{splitenergy}
  \sum_{\substack{\text{lines } l: \\ |l \cap (C \times C)| \geq 2, \\
  |l \cap (D \times D)| \geq 2}} |l \cap (C \times C)|^2|l \cap (D \times D)|^2 + \sum_{\substack{\text{lines } l: \\ |l \cap (C \times C)| = 1  \\ \text{XOR}\\
  |l \cap (D \times D)| =1 }} |l \cap (C \times C)|^2|l \cap (D \times D)|^2.\end{equation}
 To deal with the second term, note that
  \begin{align*}\sum_{\substack{\text{lines } l: \\ |l \cap (C \times C)| = 1  \\ \text{XOR}\\
  |l \cap (D \times D)| =1 }} |l \cap (C \times C)|^2|l \cap (D \times D)|^2 &= \sum_{ \substack{l: |l \cap (C \times C)| = 1 \\ |l \cap (D \times D)| \geq 2}} |l \cap (D \times D)|^2 
  + \sum_{\substack{l: |l \cap (D \times D)| = 1 \\ |l \cap (C \times C)| \geq 2}} |l \cap (C \times C)|^2 
  \\ & \ll |D|^4 + |C|^4.
 \end{align*}
  We now bound the first summand in \eqref{splitenergy}. By the Cauchy-Schwarz inequality,
  \begin{align}\label{quadruplesestimateenergy}
    \sum_{\substack{\text{lines } l: \\ |l \cap (C \times C)| \geq 2 \\
  |l \cap (D \times D)| \geq 2}} |l \cap (C \times C)|^2|l \cap (D \times D)|^2
  & \leq   \left( \sum_ {\substack{\text{lines } l: \\ |l \cap (C \times C)| \geq 2}} |l \cap (C \times C)|^4 \right) ^{1/2}  
  \left( \sum_{ \substack{\text{lines } l:  \\
  |l \cap (D \times D)| \geq 2}} |l \cap (D \times D)|^4 \right) ^{1/2}
  \end{align}
so that we may instead separately bound the number of quadruples from $C \times C$ and $D \times D$. To do this, one can dyadically decompose and apply Corollary \ref{krich} as follows:

\begin{align*}
 \sum_ {\substack{\text{lines } l: \\ |l \cap (C \times C)| \geq 2}} |l \cap (C \times C)|^4 &= \sum_{j=1}^{\log|C|} \sum_{l: 2^j \leq |l \cap (C \times C)| < 2^{j+1}} |l \cap (C \times C)|^4
 \\& \ll \sum_{j=1}^{\log|C|} \left(2^j|C|^4+|C|^22^{3j} \right)\ll |C|^5.
\end{align*}

    The same bound holds for the corresponding term for $D$ in \eqref{quadruplesestimateenergy}. Putting everything together, it follows that
    \[
    E(\cL) \ll |C|^{5/2}|D|^{5/2}+|C|^4+|D|^4,
    \]
    as required.
 \end{proof}
 
 We remark that Theorem \ref{energythm1} also holds for suitably small sets in the finite field setting, up to an addition logarithmic factor. This is because recent developments in finite field incidence theory, stemming from the work of Rudnev \cite{PointPlane}, give good bounds for the number of collinear quadruples in a Cartesian product. In particular, it follows from the work of Stevens and de Zeeuw \cite{SDZ} that $C \times C$ contain $O(|C|^5\log|C|)$ collinear quadruples. This bound is tight up to the logarithmic factor, as can be seen by counting the collinear quadruples along axis parallel lines.
 
 \subsection{Proof of Theorem \ref{energythm2}}
 \begin{proof}[Proof of Theorem 3]
 The set of lines $\cL$ is of the form
 \[
 \cL=\left\{ ( d(c - \lambda) - \mu, c) : c \in C, d \in D, d(c-\lambda)-\mu \neq 0 \right \}
 \]
  for some finite sets $C, D \subset \R$ and $\lambda$, $\mu \in \R$. Inverses of lines in $\cL$ have the form $\left(\frac{1}{ d(c-\lambda) - \mu}, \frac{-c}{ d(c-\lambda) - \mu}\right)$. The energy of $\cL$ is given by the number of solutions to the equations
 \begin{align} \label{energy1}
 \frac{d_2(c_2 - \lambda) - \mu}{d_1(c_1 - \lambda) - \mu} &=  \frac{d_4(c_4 - \lambda) - \mu}{d_3(c_3 - \lambda) - \mu},  \\ \frac{c_2 - c_1}{d_1(c_1 - \lambda) - \mu} &= \frac{c_4 - c_3}{d_3(c_3 - \lambda) - \mu}. \label{energy2}
 \end{align}
 If we have $c_1 = c_2$ then in order for the second equation to be satisfied it must be the case that $c_3 = c_4$. The first equation then becomes $\frac{d_2(c_1 - \lambda) - \mu}{d_1(c_1 - \lambda) - \mu} = \frac{d_4(c_3 - \lambda) - \mu}{d_3(c_3 - \lambda) - \mu}$, the number solutions to which is at most $|D|^3$.\footnote{One can obtain better estimates for this situation by taking more care, but since it leads only to a lower order error term in the final estimate it does not seem important.} The number of solutions to \eqref{energy1} and \eqref{energy2} with $c_1 = c_2$ is therefore at most $|C|^2 |D|^3$.

 For the remaining solutions we have $c_2 \neq c_1$, and so the common solution to \eqref{energy2} must be non-zero. Let $N$ denote the number of  solutions to
\[
\frac{c_2 - c_1}{d_1(c_1 - \lambda) - \mu} = \frac{c_4 - c_3}{d_3(c_3 - \lambda) - \mu} \neq 0.
\]

We will show that
  \begin{equation} \label{Naim}
  N \ll |C|^3|D|^{3/2}.
  \end{equation}
  It then follows that the total number of solutions to the system \eqref{energy1}, \eqref{energy2} is $O(|C|^3|D|^{5/2}+|C|^2|D|^3)$. Indeed, the number of solutions to this system for which \eqref{energy2} is non-zero is at most $N|D|$ since for each solution to \eqref{energy2} there are at most $|D|$ possible valid choices for the remaining variables $d_2,d_4 \in D$ that satisfy \eqref{energy1}.

It remains to bound $N$ as in \eqref{Naim}. In order to do so, we define the quantity 
$$n(\alpha) = \left| \left\{ (c,c',d) \in C^2 \times D : \alpha = \frac{c' - c}{d(c - \lambda) - \mu} \right\} \right|$$
so that 
\begin{equation} \label{sumforN}
N = \sum_{\alpha \neq 0} n(\alpha)^2.\end{equation}
Our method for bounding $N$ will involve decomposing this sum over rich and poor $\alpha$ in terms of $n(\alpha)$. We therefore define the set of $t$\emph{-rich} values of $\alpha$ as 
$$\Lambda_t = \left\{ \alpha \neq 0 : n(\alpha) \geq t \right\}.$$

We now aim to prove that for $t \geq  2|C|$, we have
\begin{equation}\label{lambdatbound}
|\Lambda_t| \ll \frac{|C|^4|D|^2}{t^3}.\end{equation}
We prove this using the Szemer\'{e}di-Trotter theorem. Define the line set 
$$L = \left\{ y = \frac{x - c}{d(c - \lambda) - \mu} : (c,d) \in C \times D, d(c-\lambda) -\mu \neq 0 \right\}.$$
Note that as $L$ consists precisely of inverses of the lines in $\cL$, since inverses are unique we have $|L| = |\cL| = |C||D|$. Our point set is the Cartesian product $P = C \times \Lambda_t$. Note that for any $y \in \Lambda_t$, there are at least $t$ solutions to the equation 
$$y =\frac{c' - c}{d(c - \lambda) - \mu}$$
so that the point $(c',y)$ is incident to the line $y = \frac{x - c}{d(c - \lambda) - \mu}$. This shows that we have
$$|\Lambda_t|t \leq I(P,L).$$
Bounding the other side by the Szemer\'{e}di-Trotter theorem\footnote{Here we take a little more care with the multiplicative constant in order to quickly dismiss the second error term. This precise statement can be found in \cite[Theorem 8.3]{TV}.} we have
$$|\Lambda_t|t \leq 4|\Lambda_t|^{2/3} |C|^{4/3} |D|^{2/3} + 4|C||D| + |\Lambda_t||C|.$$
Note that using $t \geq 2|C|$, the third term on the right hand side above can be discarded. Therefore,
$$
|\Lambda_t| \ll \frac{|C|^4|D|^2}{t^3} +\frac{|C||D|}{t} .$$
If $t\leq |C|^{3/2}|D|^{1/2}$ then the first term is dominant and we have the desired conclusion \eqref{lambdatbound}. However, if $t \geq |C|^{3/2}|D|^{1/2}$ then $\Lambda_t$ is empty since
\[
n(\alpha) \leq \min\{|C|^2, |C||D|\} \leq |C|^{3/2}|D|^{1/2}
\]
for all $\alpha \neq 0$. This concludes the proof of \eqref{lambdatbound}.

Let $\Delta \geq 2|C|$ be a parameter to be determined later. We decompose the sum \eqref{sumforN} and apply the bound \eqref{lambdatbound} as follows:
\begin{align} \label{Ndecomp}
    N &= \sum_{\alpha \neq 0   :   n(\alpha) < \Delta}n(\alpha)^2 +  \sum_{\alpha  \neq 0 : n(\alpha) \geq \Delta}n(\alpha)^2 \nonumber \\ & \leq   \Delta \sum_{\alpha   :   n(\alpha) < \Delta}n(\alpha) + \sum_{i\geq 0} \sum_{\substack{\alpha \neq 0 : \\ 2^i \Delta \leq n(\alpha) < 2^{i+1} \Delta}}n(\alpha)^2\nonumber \\
    & \leq \Delta |D||C|^2 + \sum_{i\geq 0} |\Lambda_{2^i \Delta} | (2^{i+1} \Delta)^2
    \\& \ll  \Delta |D||C|^2 + \sum_{i\geq 0} \frac{|C|^4|D|^2}{ (2^{i} \Delta)}
    \\& \ll \Delta |D||C|^2 +  \frac{|C|^4|D|^2}{ \Delta}.
\end{align}

We now optimise our choice of $\Delta$ to balance these terms; we set $\Delta = |C||D|^{1/2}$ (so the assumption $\Delta \geq 2|C|$ is valid as long as $|D| \geq 4$) and find the bound
$$N \ll |C|^3|D|^{3/2}.$$
This completes the proof of \eqref{Naim}, and thus also the proof of Theorem \ref{energythm2}.
 \end{proof}

 \section{A New Geometric Problem} \label{sec:geometric}
 
 In this section we introduce a new geometric problem and prove some first results towards a more general conjecture. As a corollary of one of these results, a proof of Theorem \ref{expander} is given.
 
 
 Let $P \subset \R^2$ be a finite set of non-collinear points, and define $L(P)$ to be the set of lines defined by $P$, that is, all lines which contain at least two points from $P$. A classical question in discrete geometry is that of determining the minimum possible number of directions defined by $P$. An optimal bound for this question was given by Ungar \cite{Ungar}, who proved that a set of $2N$ non-collinear points determine at least $2N$ directions.

We may consider this result as concerning the number of intersections of $L(P)$ with the projective line at infinity in $P^2(\R)$. In principle there is nothing special about this line at infinity, and we can ask about the number of intersections of $L(P)$ with any arbitrary line. Indeed, after applying a projective transformation to $P$, Ungar's Theorem implies that, for any non-collinear set $P$ of cardinality $2N$ and any line $l$ in the plane such that $l \cap P = \emptyset$, $L(P)$ intersects $l$ in at least $2N$ points. 

Making a small abuse of notation, we write the set of points where lines from $L(P)$ intersect a fixed line $l$ as $L(P) \cap l$. 

The following question then arises; if we take \emph{two} arbitrary lines $l_1$ and $L_2$, can both $L(P) \cap l_1$ and $L(P) \cap l_2$ achieve the minimum value of $\Theta(|P|)$? This question has a sum-product type flavour; we may expect that minimising the intersections of $L(P)$ with one line necessarily makes the intersection with any other line large.

There are some degenerate situations that must first be ruled out. For example, if $P$ contains a rich line, it may be the case that $L(P)$ itself has size $\Theta(|P|)$. Consider, for example, the case when $P$ consists of points on a line and a single point off the line. There are other degenerate cases whereby $L(P) \cap l_1$ and $L(P) \cap l_2$ both have linear size in $|P|$, but such examples that we are aware of come from sets which contain many points on a single line. So, we need to rule out the case when $P$ contains very rich lines. Doing so, we have arrived at the following conjecture.



\begin{Conjecture} \label{conj}
For all $\delta>0$ there exists $\epsilon=\epsilon(\delta)>0$ such that the following holds. Let $P \subset \R^2$ be a finite set of points with the property that no more than $|P|^{1-\delta}$ points of $P$ lie on a line, and let $l_1$ and $l_2$ be arbitrary lines with $l_1 \neq l_2$. Then
$$|L(P) \cap l_1| + |L(P) \cap l_2| \gg |P|^{1 + \epsilon}.$$
\end{Conjecture}


One way to simplify the conditions of the conjecture is to impose the restriction that $P$ takes the form of a balanced grid. This is what we do for the remainder of the section, taking $P=A \times B$ with $A$ and $B$ having comparable sizes. In such cases, we can apply the energy bounds from the previous section in order to prove the following two cases of Conjecture \ref{conj}, where the two lines $l_1$ and $l_2$ have additional restrictions.


\begin{Theorem}\label{pencilstheoremaxes}
Let $P=A \times B$ with  $A, B \subset \mathbb R^*$ such that $|B| \leq |A|^2$. Let $l_1$ denote any vertical or horizontal line not defined by $L(P)$, and let $l_2$ denote any affine line not parallel to $l_1$. Then
 $$\max \{ |L(P) \cap l_{1}|, |L(P) \cap l_2| \} \gg |A||B|^{15/14}.$$
\end{Theorem}

If one of the fixed lines is $l_{\infty}$, Theorem \ref{energythm2} can be used to give a further quantitative improvement.

\begin{Theorem}\label{pencilstheorem}
 Let $P = A \times B$ be a finite Cartesian product, let $l_{\infty}$ denote the line at infinity, and let $l_1$ be an arbitrary non-vertical and non-horizontal line. If $|A|^{5/3} \geq |B| \geq |A|^{3/5}$, then we have
 $$\max \{ |L(P) \cap l_{\infty}|, |L(P) \cap l_1| \} \gg |A|^{15/14}|B|^{15/14}.$$
\end{Theorem}

In fact, we are not aware of a single example of a point set $P=A \times B$ (with $A$ and $B$ having comparable size) and an affine line $l$ such that $|L(P) \cap l|=\Theta(|P|)$, and it is plausible that there is some $\epsilon>0$ such that the bound
\[
|L(P) \cap l| \gg |P|^{1+\epsilon}
\]
is always true in such a case.

If Conjecture \ref{conj} is to believed, an obvious question to ask is how large can $\epsilon$ be? Note that, if $P=A \times B$ with $A=\{1,2,\dots, N^{1-\delta}\}$ and $B=\{1,2,\dots,N^{\delta}\}$ then $L(P)$ intersects the line at infinity in $\Theta(N)$ points and the $y$-axis in $O(N^{1+\delta})$ points, which implies that $\epsilon$ cannot be taken to be larger than $\delta$.
 
 \subsection{Proof of Theorem \ref{pencilstheoremaxes} and Theorem \ref{expander}}

\begin{proof}[Proof of Theorem \ref{pencilstheoremaxes}]

WLOG we may assume that $l_1$ is a horizontal line, by reflecting the plane in the line $y =x$ if needed. We shall apply a projective transformation $\pi$ to the plane, sending the line $l_1$ to $l_{\infty}$, and sending the line $l_2$ to the $y$ axis. We do this by defining the image of three points.
\begin{align*}
    \pi([1;0;0]) &= [1;0;0] \quad \text{Horizontal direction is sent to itself.} \\
    \pi(l_1 \cap l_2) &= [0;1;0]  \quad \text{Intersection of $l_1$ and $l_2$ sent to the vertical direction.}\\
    \pi(l_2 \cap l_{\infty}) &= [0;0;1] \quad \text{Slope of $l_2$ is sent to the origin.}
\end{align*}
Note that the three points being mapped above are distinct; $l_1$ and $l_2$ intersect at an affine point, and $l_2$ is not horizontal.

We shall now assume that $l_2$ is not vertical; the case where $l_2$ is vertical is slightly simpler, and we therefore omit it. Assuming the line $l_1$ has equation $y = \alpha$, and $l_2$ has equation $y = \beta x + \gamma$, the projective transformation takes the form
$$\pi = \begin{pmatrix} 1 & \frac{-1}{\beta} & \frac{\gamma}{\beta} \\ 0 & 0 & 1 \\ 0 & \frac{1}{\beta} & \frac{-\alpha}{\beta}
\end{pmatrix}.$$
After the application of $\pi$, the point set $P$ is sent to the new point set $P' := \pi(P)$. As projective transformations preserve incidence structure, it follows that
\[
|L(P') \cap l_{\infty}|=|L(P) \cap l_1|, \,\,\,\,\,\,\,\,\,\,\, |L(P') \cap l_{y}|=|L(P) \cap l_2|.
\]
Therefore, we seek to bound
\[
\max \{|L(P') \cap l_{\infty}| , |L(P') \cap l_y|\}.
\]
Our first task is to find the form of points in $P'$. As the original point set $P$ was given by the intersection of two pencils of lines (one vertical and one horizontal pencil), the points in $P'$ are also given by the intersection of two pencils. From the way we defined $\pi$, we see that the horizontal pencil of lines of the form $y = b$ are sent to horizontal lines of the form $y = \frac{\beta}{b - \alpha}$. Lines of the form $x = a$ in the vertical pencil are sent to lines of the form $y = \frac{(x+1)\beta}{a\beta + \gamma - \alpha}$.\footnote{We may need one element from $A$ and $B$ respectively in order to avoid dividing by zero. This can be done at the outset of the proof and the details are omitted.}  This can be checked by calculating the images of such lines under $\pi$. The point set $P'$ is therefore given by the intersection points of the two pencils of lines $$y = \frac{(x+1)\beta}{a\beta + \gamma - \alpha}, \quad y = \frac{\beta}{b - \alpha} \quad a \in A, b \in B.$$
Therefore,
$$P' = \left\{ \left( \frac{a\beta + \gamma - \alpha}{b - \alpha} - 1, \frac{\beta}{b - \alpha} \right) : (a,b) \in A \times B \right\}.$$
To simplify this, we define
$$A' = A + \frac{\gamma - \alpha}{\beta} ,\qquad B' = \frac{\beta}{B - \alpha}.$$
In particular we note that $P'$ has the form $(a'b' - 1, b')$ with $(a',b') \in A' \times B'$, and we have $|A| = |A'|$, $|B| = |B'|$.

Define
\[
\cL:=\{l \in L(P') : \text{ $l$ is not horizontal or vertical}\}.
\]
Using the notation of the affine group, each element of $l \in \cL$ can be expressed as a pair $(s,t) \in \mathbb R^* \times \mathbb R$, where $s$ is the slope of $l$ and $t$ is its $y$-axis intercept. We now `complete' $\cL$ into a Cartesian product $S \times Y$, where $S \subset \mathbb R^*$ is the set of slopes defined occuring in $\cL$ and $Y\subset \mathbb R$ is the set of $y$-axis intercepts occuring. Critically, we have
\begin{equation} \label{somefacts}
    |S| \leq  |L(P') \cap l_{\infty}| = |L(P) \cap l_1|, \quad |Y| \leq |L(P') \cap l_y|= |L(P) \cap l_2|.
    \end{equation}
By Beck's theorem, we know that $P'$ defines many lines; $|L(P')| \gg |A|^2|B|^2$. At most $|A||B|+|B|$ of these lines are horizontal or vertical, and so
\[
|\cL|=|L(P')|- |\{ \text{all horizontal or vertical lines in $L(P')$}\}| \gg|A|^2|B|^2.
\]
Since each line in $\cL$ intersects $P'$ in at least two places, we have
\begin{equation} \label{incehalf}
    |A|^2|B|^2 \ll 2|\cL| \leq I(P', \cL) \leq I(P', S \times Y).
\end{equation}
To upper bound the number of incidences $I(P', S \times Y)$, we use the observation that 
$$I(P',S \times Y) = I(S \times Y, L')$$
where in the right hand side of the above equality, $S \times Y$ is a Cartesian product of points, and $L'$ is the set of lines of the form $y = -g_1x + g_2$ such that $(g_1,g_2) \in P'$. To see this, note that an incidence contributing to $I(P', S \times Y)$ has the form $g_2 = sg_1 + t$ for $(g_1,g_2) \in P'$, $(s,t) \in S \times Y$. Writing this as $t = -g_1s + g_2$, we see that it is also an incidence between the point $(s,t) \in S \times Y$, and the line $y = -g_1x + g_2$. 

We shall now use Theorem \ref{mishailya8} to bound $I(S \times Y, L')$. As lines in $L'$ have the form $(-a'b'-1,b')$ with $(a',b') \in A' \times B'$, we are in the scenario to apply Theorem \ref{energythm2}, which gives 
$$E(L') \ll |A|^3 |B|^{5/2} +  |A|^2 |B|^3 \ll |A|^3|B|^{5/2},$$
where the latter inequality uses the assumption that $|B| \leq |A|^2$.
Plugging this into Theorem \ref{mishailya8} and applying \eqref{incehalf} gives
$$|A|^2|B|^2 \ll I(S\times Y, L') \ll |Y|^{1/2}|S|^{2/3}|A|^{5/6}|B|^{3/4} + |Y|^{1/2}|A||B|.$$
If the second term dominates we have $|Y| \gg |A|^2|B|^2$, which is better than required. We may therefore assume the leading term is dominant, giving $|S|^8|Y|^6 \gg |A|^{14}|B|^{15}$ and thus
$$ \max \{|S|, |Y|\} \gg |A||B|^{15/14}.$$
Recalling \eqref{somefacts}, the proof is complete.
\end{proof}

Theorem \ref{expander} now follows as an immediate corollary.

\begin{proof}[Proof of Theorem \ref{expander}]

Define
$$Q[A] := \left\{ \frac{a_1a_4 - a_2a_3}{a_1-a_3} : a_i \in A \right\}.$$
Let $P= A \times A$, and consider the line set $L(P)$. The line connecting $(a_1,a_2)$ and $(a_3,a_4)$ with $a_1 \neq a_3$ has equation
$$y = \frac{a_2 - a_4}{a_1 - a_3}x + \frac{a_1a_4 - a_2a_3}{a_1 - a_3}.$$
We now have two observations. Firstly, the size of $Q[A]$ is precisely the number of $y$ intercepts defined by $L(P)$. Secondly, the point set $A \times A$ is symmetric in the line $y=x$, so that the number of $y$ intercepts defined by $L(P)$ is precisely the number of $x$ intercepts defined by $L(P)$. That is,
$$|Q[A]|=|L(P) \cap l_y| = |L(P) \cap l_x| .$$
Theorem \ref{pencilstheoremaxes} therefore implies that
\[
|Q[A]| \gg |A|^{2+\frac{1}{14}}.
\]
\end{proof}

\subsection{Proof of Theorem \ref{pencilstheorem}}


\begin{proof}[Proof of Theorem \ref{pencilstheorem}]
Firstly note that if $|l_1 \cap (A \times B)| \geq 2$, then $l_1 \cap L(P)$ is in fact an infinite set and we are done. So, we can assume that $|l_1 \cap L(P)| \leq 1$. In fact, we make the simplifying assumption that $l_1 \cap (A \times B) =\emptyset$. If this is not the case then we can delete one point from $A \times B$ at the outset and make some small modifications to the forthcoming argument, but we omit these details.

We apply a projective transformation to the plane, sending the line $l_1$ to the line at infinity, and sending $l_{\infty}$ to the $y$ axis. Let $l_1$ be the line $y = \lambda x + \mu$, with $\lambda \neq 0$. We define the images
\begin{align*}
    \pi([1;\lambda ;0]) &= [0;1;0] \quad \text{Slope of $l_1$ is sent to vertical direction.} \\
    \pi([0;\mu;1]) &= [1;0;0]  \quad \text{Intercept of $l_1$ is sent to horizontal direction}\\
    \pi([0;1;0]) &= [0;0;1] \quad \text{Vertical direction is sent to the origin.}
\end{align*}
This gives the projective transformation
$$\pi = \begin{pmatrix} 0 & 0 & 1 \\ 1 & 0 & 0 \\ -\lambda & 1 & -\mu
\end{pmatrix}.$$
Under $\pi$, the Cartesian product $A \times B$ is mapped to the intersection points of two pencils of lines. Write $P'= \pi(P)$. Note that since $l_1 \cap (A \times B) =\emptyset$, $P'$ consists of only affine points. As projective transformations preserve incidence structure, we have
$$|L(P') \cap l_{\infty}| = |L(P) \cap l_1|  , \quad   |L(P') \cap l_{y}| = |L(P) \cap l_{\infty}|.$$
Therefore, as in the proof of Theorem \ref{pencilstheoremaxes}, we seek to bound
\[
\max \{|L(P') \cap l_{\infty}| + |L(P') \cap l_y|\}.
\]
Our first task is to find the form of points in $P'$. By the definition of $\pi$, the vertical direction was sent to the origin. Therefore the pencil corresponding to the set $A$ is now centred at the origin. Call this pencil $P_1$. It contains lines of the form $y = ax$ for $|A|$ values of $a$. The second pencil, call it $P_2$, lies somewhere on the $y$ axis, call it $(0,\alpha)$, and thus contains lines of the form $y-\alpha = bx$ for $|B|$ values of $b$. We do not necessarily have here that $a \in A$ or $b \in B$, only that they come from sets (say $A'$ and $B'$) with $|A'|= |A|$ and $|B'| = |B|$. The intersection points are therefore 
$$P' = \left(\frac{\alpha}{a-b}, \frac{ \alpha a}{a-b} \right), \quad (a,b) \in A' \times B'.$$
All these points are affine (i.e. we do not have $a=b$, since $P'$ contains no points at infinity).


Define
\[
\cL=\{l \in L(P) : \text{ $l$ is not horizontal or vertical}\}.
\]
By Beck's Theorem we have $|L(P')| \gg |A|^2|B|^2$, since there are at most $\max \{ |A|,|B| \}$ collinear points. Furthermore, as in the proof of Theorem \ref{pencilstheoremaxes}, the number of horizontal and vertical lines in $L(P')$ is negligible, and so
\[
|\cL| \gg |A|^2|B|^2.
\]
Let $S$ denote the set of slopes of lines present in $\cL$, and let $Y$ denote the set of $y$ intercepts present in $\cL$. We complete the set $\cL$ into the Cartesian product $S \times Y$, of lines of the form $y =sx + t$, $(s,t) \in S \times Y \subset \mathbb R^* \times \mathbb R^*$. Firstly, note that we have
$$|A|^2|B|^2 \ll 2|\cL|\leq I(P', \cL) \leq I(P', S \times Y)$$
Secondly, note that 
\begin{equation}\label{somefacts2}
|Y| = |\cL \cap l_y| \leq |L(P') \cap l_y|, \quad |S| = |\cL \cap l_{\infty}| \leq |L(P') \cap l_{\infty}|.
\end{equation}
Consider an incidence contributing to $I(P',S \times Y)$. This corresponds to a point $(p_1,p_2) \in P'$ lying on the line $y = sx + t$, so that $p_2 = sp_1 + t$. By the simple rearrangement $t = -p_1 s + p_2$, we see that the \emph{point} $(s,t) $ lies on the \emph{line} $y = -p_1x + p_2$. From this observation, we have
$$I(P',S \times Y) = I(S \times Y, L')$$
where $L'$ is the set of lines of the form $y = -p_1x + p_2$ such that $(p_1,p_2) \in P'$, and $S \times Y$ is a Cartesian product of points. The lines $L'$ are therefore of the form $\left(\frac{-\alpha}{a-b}, \frac{ \alpha a}{a-b} \right)$. The energy of $L'$ can then be bounded by Theorem \ref{energythm1} as
$$E(L') \ll |A|^{5/2}|B|^{5/2} + |A|^4 + |B|^4.$$
The condition $|A|^{5/3} \geq |B| \geq |A|^{3/5}$ ensures that the leading term above dominates. Plugging this energy bound into Theorem \ref{mishailya8} yields
$$|A|^2|B|^2 \ll I(S \times Y, L') \ll |S|^{2/3}|Y|^{1/2}|A|^{3/4}|B|^{3/4} + |Y|^{1/2}|A||B|.$$
If the second term dominates, we get $|Y| \gg |A|^2|B|^2$, better than needed. We then assume the leading term dominates, which upon rearrangement gives
$$|S|^8|Y|^6 \gg |A|^{15}|B|^{15}.$$
Recalling \eqref{somefacts2}, the proof is complete.
\end{proof}

\section{Asymmetric `few sums many products' problem} \label{sec:FSMP}

The purpose of this section is to prove Theorem \ref{thm:FSMP}. We will need a small improvement on the bound \eqref{RSEnergy}, for the energy of lines of the form $(a,ab)$. In fact, such a result is already provided in \cite[Lemma 21]{RuSh}, which we state below.
\begin{Theorem} \label{thm:RSEnergyBetter}
For all $\kappa>0$ there exists $\delta=\delta(\kappa)>0$ such that, for all $C,D \subset \mathbb R^*$ with $|D|^{\kappa} \leq|C|\leq |D|^2$, the set of lines
\[
\cL=\{(c,cd) : (c,d) \in C \times D \} \subset \text{Aff}(\mathbb R)
\]
satisfies the bound
\[
E(\cL) \ll |C|^{\frac{5}{2}-\delta}|D|^3.
\]
\end{Theorem}

The proof of Theorem \ref{thm:RSEnergyBetter} is significantly more difficult than that of \eqref{RSEnergy}, utilising bounds on growth in the affine group proved elsewhere in \cite{RuSh}, as well as an additive combinatorial tool due to Shkredov \cite{Sh2} which gives structural information for a set when its second moment and third moment energy are in a particular `critical case'.

Theorem \ref{thm:RSEnergyBetter} can then be combined with Theorem \ref{mishailya8} to give the following improvement to \eqref{RSgrids}.

\begin{Theorem} \label{thm:GridsBetter}
For all $\kappa>0$ there exists $\delta=\delta(\kappa)>0$ such that, for all $C,D \subset \mathbb R^*$ with $|D|^{\kappa} \leq|C|\leq |D|^2$, the set of lines
\[
\cL=\{(c,cd) : (c,d) \in C \times D \} \subset \text{Aff}(\mathbb R)
\]
satisfies the bound
\[   I(A \times B, \cL) \ll |A|^{2/3}|B|^{1/2}|C|^{3/4-\delta}|D|^{5/6}+|B|^{1/2}|C||D|
\]
for any $A,B \subset \mathbb R$.
\end{Theorem}

We will also need to know that sets with small sum set have superquadratic sized triple product sets. The precise statement we use is \cite[Theorem 3.2]{RNS}, stated below.

\begin{Theorem} \label{thm:FSMTP}
There is an absolute constant $C>0$ such that, for any $A \subset \mathbb R$ and $|A+A|=K|A|$ we have
\[
|AAA| \gg \frac{|A|^{2+\frac{1}{392}}}{K^{\frac{125}{56}}(\log|A|)^C}.
\]
\end{Theorem}

We are now ready to begin the proof of Theorem \ref{thm:FSMP}, which we restate below for convenience.

\begin{Theorem**} 

For all $\kappa>0$ there exists $k=k(\kappa)>0$ such that for all $A,B \subset \mathbb R$ with $|A|^{\kappa} \leq |B|$ and writing $|A+A|=K|A|$, we have
\[ |AB| \gg \frac{|A||B|^{1/2+k}}{K^{4/3}}.
\]
\end{Theorem**}

\begin{proof}
Fix $\kappa >0$ and let $\delta=\delta(\kappa)$ be the value given by the statement of Theorem \ref{thm:GridsBetter}. Firstly, let us assume that $|A|^{\kappa} \leq|B|\leq |A|^{2-c_0}$, where $c_0=\min\{\frac{8\delta}{1+4\delta}, \frac{1}{392}\}$. For this case, we can use a simple Elekes-type argument in combination with Theorem \ref{thm:GridsBetter}.

Let $\cL$ be the set of all lines of the form $y=c(x-d)$ with $c \in B$ and $d \in A$. Using the notation of $\text{Aff}(\mathbb R)$, this is the set of all lines of the form $(c,-cd)$. Define $\cP=(A+A) \times (AB)$. Observe that, for each $a \in A$ and for any $(c,d) \in A \times B$ we have an incidence between the point $(a+d,cd)$ and the line $(c,-cd)$. Applying this observation and Theorem \ref{thm:GridsBetter} yields
\begin{align*}
|A|^2|B|\leq I(\cP, \cL) &\ll |A+A|^{2/3}|AB|^{1/2}|B|^{3/4-\delta}|A|^{5/6}+|AB|^{1/2}|B||A|
\\& \ll |A+A|^{2/3}|AB|^{1/2}|B|^{3/4-\delta}|A|^{5/6},
\end{align*}
where the latter inequality uses the assumption that $|B|\leq |A|^{\frac{2}{1+4\delta}}$. Writing $|A+A|=K|A|$ and rearranging gives
\begin{equation} \label{conc1}
|AB| \gg \frac{|A||B|^{\frac{1}{2}+2\delta}}{K^{\frac{4}{3}}}.
\end{equation}

Note also that Theorem \ref{thm:FSMP} is true for trivial reason when $|B| \geq |A|^3$. This is simply because, for any sets $A$ and $B$ with this property
\[
|AB| \geq |B| \geq |A||B|^{\frac{2}{3}}.
\]

Now we turn to the case when $ |A|^{2-c_0} \leq |B| \leq |A|^3$. Applying Theorem \ref{thm:FSMTP} and the Pl\"{u}nnecke-Ruzsa Theorem, gives
\[
\frac{|AB|^3}{|B|^2} \geq |AAA| \gg \frac{|A|^{2+\frac{1}{392}}}{K^{\frac{125}{56}}(\log|A|)^C}.
\]
Rearranging and using the bound $|B| \geq |A|^{2-c_0}$, it follows that
\[
|AB| \gg \frac{|B|^{\frac{2}{3}}|A|^{\frac{2}{3}+ \frac{1}{1176}}}{K^{\frac{125}{168}} (\log|A|)^C} 
\gg \frac{|B|^{\frac{1}{2}}|A|^{1+ \frac{1}{1176}-\frac{c_0}{6} }}{K^{\frac{125}{168}} (\log|A|)^C} 
\geq \frac{|B|^{\frac{1}{2}}|A|^{1+ \frac{1}{2352}}}{K^{\frac{125}{168}} (\log|A|)^C}
\gg \frac{|B|^{\frac{1}{2}}|A|^{1+ \frac{1}{2500}}}{K^{\frac{4}{3}} }.
\]
In the penultimate inequality above we have used the fact that $c_0 \leq \frac{1}{392}$, while the last step is trivially true (for $A$ sufficiently large) and is carried out only to simplify the expression. Finally, since $|A| \geq |B|^{1/3}$, we conclude that
\begin{equation} \label{conc2}
|AB| \gg \frac{|B|^{\frac{1}{2}+\frac{1}{7500}}|A|}{K^{\frac{4}{3}} }.
\end{equation}
Examining inequalities \eqref{conc1} and \eqref{conc2}, we see that the proof is complete by taking $k=\min \{2\delta, \frac{1}{7500}\}$.

\end{proof}

To conclude this section, we remark that the arguments from proof of Theorem \ref{thm:RSEnergyBetter} can be applied to bound the energy of lines of the form $(cd,c)$ thus giving an analogous small improvement to Theorem \ref{energythm2}. One can then repeat some arguments from the proof of Theorem \ref{thm:FSMP} above but using this improved energy bound in order to get a threshold breaking sum-product type result of the following form:
\[
|AA| \leq K|A| \Rightarrow |(A+1)B| \gg_K|A||B|^{\frac{1}{2}+c}.
\]
However, this result does not apply in the endpoint case when the size of $B$ is very close to $|A|^2$, since we do not have a suitable analogue of Theorem \ref{thm:FSMTP} for this problem. It may be an interesting research problem to prove such a result, which take the form of the following superquadratic growth statement:
\[
|AA| \leq K|A| \Rightarrow |(A+1)(A+1)(A+1)| \gg |A|^{2+c}.
\]

\section{Further Expander Results} \label{sec:other}
In this section we explore some further implications of the new techniques introduced in \cite{RuSh} and developed further in this paper. 

In the following, the symbol $\lesssim$ is used to absorb powers of $\log|A|$. That is, $X \lesssim Y$ if there are absolute constants $c,d > 0$ such that $X \leq c Y (\log|A|)^d$. 

\subsection{The size of $AA+A$}
Our first application concerns the size of the set $AA+A$, which is perhaps the most obvious set one can construct by a combination of additive and multiplicative operations. The `threshold' bound $|AA+A| \gg |A|^{3/2}$ follows from a simple application of the Szemer\'{e}di-Trotter Theorem. In \cite{RNRSS}, a fairly involved argument based upon the geometric setup of Solymosi \cite{So} was used to give the small improvement $|AA+A| \gg |A|^{3/2+c}$, with $c=2^{-222}$. This argument required the restriction that $A$ consists only of positive reals.

In \cite{RuSh}, the authors observed that they could remove this restriction and implicitly give a better value of $c$, although they did not calculate $c$ explicitly. They also obtained an improvement to the corresponding threshold energy bound.

The following result gives the first `reasonable' explicit bound for this set.

\begin{Theorem}
Let $A \subset \R$ be finite. Then
$$|AA+A| \gtrsim |A|^{3/2 + 1/194}.$$
\end{Theorem}
The proof uses a combination of Theorem \ref{mishailya8} and the theory of the quantity $d_*(A)$, in addition to additive and multiplicative energies. The additive energy of a set $A$ is defined as follows.
$$E^+(A) := \left| \left\{(a,b,c,d) \in A^4 : a+b = c+d  \right\} \right|.$$
For an integer $k$, the $k$'th multiplicative energy of a set $A$ is defined as 
$$E_k^*(A) := \left| \left\{ (a_1,a_2,...,a_{2k}) \in A^{2k} : \frac{a_1}{a_2} = \frac{a_3}{a_4} = ... = \frac{a_{2k-1}}{a_{2k}}  \right\} \right|.$$ In the case $k=2$ this is simply called the multiplicative energy and denoted by $E^*(A)$. We define
$$d_*(A) := \min_{t > 0} \min_{Q \neq \emptyset, R \subset \R \setminus \{0\}} \frac{|Q|^2|R|^2}{|A|t^3}.$$ The following result is stated in \cite{Var2}, and proved in a different form in \cite{Sh2}.
\begin{Theorem}\label{denergy}
Let $A \subset \R$ be a finite set. Then we have
$$E^+(A) \lesssim d_*(A)^{7/13} |A|^{32/13}.$$
\end{Theorem}
A second result concerning $d_*(A)$ is the following decomposition type theorem, which can be viewed as a refinement of the Balog-Szemer\'{e}di-Gowers Theorem tailored towards a specific sum-product application. See \cite[Lemma 6.4]{Var2}.
\begin{Lemma}\label{ddecomp}
Let $A \subset \R$ be finite, and let $E^*(A) \gg \frac{|A|^3}{K}$. Then there exists $A' \subseteq A$ and a number $|A| \geq \Delta \gg |A|/K$ such that 
$$|A'| \gtrsim \frac{|A|^2}{K \Delta}, \qquad d_*(A') \lesssim \frac{K|A'|^2}{|A|\Delta}.$$
\end{Lemma}
We are now ready to begin the proof.
\begin{proof}
It can be assumed without loss of generality that $0 \notin A$. Write $l_{a,b}$ for the line with equation $y=ax+b$ and define
\[
\cL=\{ l_{a,b} : a,b \in A\}
\]
We begin by bounding the energy of $\cL$. 
The following refinement of \eqref{RSEnergy} was given in \cite{RuSh}:
$$E(\cL) \leq E^*_4(A)^{1/2}Q^{1/2}.$$
Here $Q$ denotes the number of solutions to the equation
\[
\frac{a_1-a_2}{a_3-a_4}=\frac{a_5-a_6}{a_7-a_8}, \,\,\,\,\,\,\, a_i \in A.
\]
The bound $Q \ll |A|^6\log|A|$ was established in \cite{rectangles} (a simpler proof was later given in \cite{MRNS} - see also \cite{BW} for another presentation of this proof).

Using this result and the trivial bound $E^*_4(A) \leq |A|^2 E^*(A)$, $E(\cL)$ may then be bounded by
$$E(\cL) \lesssim |A|^4 E^*(A)^{1/2}.$$
We now use Theorem \ref{mishailya8}. Define $P = A \times (AA + A)$. Since for every $c \in A$ we have the incidence $(c,ac+b) \in l_{a,b}$, it follows that there are at least $|A|^3$ incidences. We bound the other side via Theorem \ref{mishailya8} as
$$|A|^3 \leq I(P,\cL) \ll |AA+A|^{1/2} |A|^2 E^*(A)^{1/12} + |AA+A|^{1/2}|A|^2.$$
Note that we may assume the leading term dominates, as otherwise we do better. We then have
$$|AA+A| \gg \frac{|A|^2}{E^*(A)^{1/6}}.$$
Now define $K$ via $E^*(A) = K^{-1}|A|^3$, so $1 \leq K \leq |A|$. The above inequality gives
\begin{equation}\label{AA+A8}|AA+A| \gg K^{1/6}|A|^{3/2}.\end{equation}
We shall now apply Lemma \ref{ddecomp} to $A$, to find a subset $A'$ with $|A'| \gtrsim \frac{|A|^2}{K \Delta}$ and with $d_*(A') \lesssim \frac{K|A'|^2}{|A|\Delta}$ for some $|A| \geq \Delta \gg |A|/K$. Applying Theorem \ref{denergy} gives
$$\frac{E^+(A')^{13/7}}{|A'|^{32/7}} \lesssim  d_*(A') \lesssim  \frac{K|A'|^2}{|A|\Delta} $$ and so
$$E^+(A') \lesssim \frac{K^{7/13}|A'|^{46/13}}{\Delta^{7/13}|A|^{7/13}}.$$
It follows from Cauchy-Schwarz that for any $a \in A'$ we have 
$E^+(A') \geq E^+(A',aA') \geq \frac{|A'|^4}{|A'+ aA'|}.$ Noting that $|A'+aA'| \leq |AA +A|$ then gives
$$|AA+A| \gtrsim \frac{|A'|^{6/13} |A|^{7/13} \Delta^{7/13}}{K^{7/13}} \gtrsim \frac{|A|^{19/13}\Delta^{1/13}}{K} \gg \frac{|A|^{20/13}}{K^{14/13}}.$$
Now note that estimate \eqref{AA+A8} improves as $K$ increases, and the estimate above improves as $K$ decreases. We therefore find the value of $K$ where the two estimates give the same value, as elsewhere we always have a better result. We have
$$\frac{|A|^{20/13}}{K^{14/13}} = K^{1/6}|A|^{3/2} \implies K = |A|^{3/97}.$$
Plugging this back into the estimates gives
$$|AA+A| \gtrsim |A|^{3/2 + 1/194}.$$
as required.
\end{proof}
\subsection{Another Three-Variable Expander}
A further application of the line energy method gives the following expander result.
\begin{Theorem} \label{thm:3var}
Let $A \subseteq \R$ be a finite set. Then we have
$$\left| \left\{ (a_1 - a_2)a_3 + a_1: a_1,a_2,a_3 \in A \right\} \right| \gg |A|^{5/3}.$$
\end{Theorem}

Note that this is better than the usual `threshold' estimate for three-variables that one obtains from a simple application of the Szemer\'{e}di-Trotter Theorem, where an exponent $3/2$ typically appears.

\begin{proof}
Theorem \ref{thm:3var} is a consequence of an energy bound on the set of lines 
\[
\cL:=\{(c-d,c) : (c,d) \in C \times D, c \neq d \}
\]
 for arbitrary finite sets $C,D \subset \R^*$. In fact, these lines are precisely the inverses of the lines given in Theorem \ref{energythm1}. It is now enough to see that for a set of (non-vertical, non-horizontal) lines $L$, we have $E(L) = E(L^{-1})$. Indeed, this is seen via
 $$l_1^{-1}l_2 = l_3^{-1}l_4 \iff l_2l_4^{-1} = l_1l_3^{-1}.$$
 Using Theorem \ref{energythm1}, we have the bound
\begin{equation} \label{energyunproven}
E(\cL) \ll  |C|^{5/2}|D|^{5/2} + |C|^4 + |D|^4.
\end{equation}
The lines $l_{c,d}\in \cL$ have the form $y = (c-d)x + c$. Let $A \subset \R$ be a finite set, and define
\[
S:=\{ (c-d)a + c : (a,c,d) \in A \times C \times D \}.
\]
Define the point set $P$ as $A \times S$. Since for each line $l_{c,d}$ in $\cL$ and each $a \in A$ we have the incidence $ (a,(c-d)a + c) \in l_{c,d}$, it follows that $I(P, \cL) \geq |A||C||D|$.
Using \eqref{energyunproven} and Theorem \ref{mishailya8} then gives
\begin{equation*}
|A||C||D| \leq I(P ,\cL) \ll |S|^{1/2}\bigg[ |A|^{2/3}  |C|^{3/4}|D|^{3/4}  + |A|^{2/3}|C||D|^{1/3} +  |A|^{2/3}|D||C|^{1/3} + |C||D| \bigg]. 
\end{equation*}
Taking $A = B = C$ ensures that the leading term dominates, completing the proof.
\end{proof}

 \section*{Acknowledgements}
The authors were partially supported by the Austrian Science Fund FWF Project P 30405-N32. We are very grateful to Mehdi Makhul for several helpful conversations, and in particular for offering important ideas towards the proof of Theorem \ref{expander}. We thank Brendan Murphy, Misha Rudnev and Ilya Shkredov for helpful discussions.

\end{document}